\documentclass[11pt,bezier]{article}
\usepackage{amsmath, amsthm, amssymb, amsfonts, graphicx,tikz}

\newtheorem{thm}{Theorem}
 \newtheorem{conj}[thm]{Conjecture}
 \newtheorem{lem}[thm]{Lemma}
 \theoremstyle{definition}
 
\newcommand{\al}{\alpha}
\newcommand{\cpg}{{\rm CPG}}

\newcommand{\va}{\vartheta}
\renewcommand{\thefootnote}

\title{Hamiltonicity of a coprime graph}

\author{ M. H. Bani Mostafa A.$^{\,\rm a}$ \qquad  Ebrahim Ghorbani$^{\,\rm a,b,}$\thanks{Corresponding author} \\[.4cm]
{\sl $^{\rm a}$Department of Mathematics, K. N. Toosi University of Technology,}\\
{\sl P. O. Box 16765-3381, Tehran, Iran}\\
{\sl $^{\rm b}$University of Hamburg, Department of Mathematics, Bundesstraße 55 (Geomatikum),}\\
{\sl 20146 Hamburg, Germany }}

\begin{document}
\maketitle
\footnotetext{{\em E-mail Addresses}: {\tt m.bani.mostafa@ut.ac.ir, ebrahim.ghorbani@uni-hamburg.de} }

\vspace{5mm}

\begin{abstract}

The  $k$-coprime graph of order $n$ is the graph with vertex set  $\{k, k+1, \ldots, k+n-1\}$ in which two vertices  are adjacent if and only if they are coprime.   We characterize Hamiltonian $k$-coprime graphs. As a particular case, two conjectures by Tout, Dabboucy, Howalla (1982) and by Schroeder (2019) on prime labeling of $2$-regular graphs follow. A prime labeling of a graph with $n$ vertices is a labeling of  its vertices with
distinct integers from $\{1, 2,\ldots , n\}$ in such a way that the labels
of any two adjacent vertices are relatively prime.

\vspace{5mm}
\noindent {\bf AMS Mathematics Subject Classification\,(2010):} 05C78, 11B75
\end{abstract}

\vspace{5mm}

\section{Introduction}
Let $G$ be a simple graph with $n$ vertices.
A {\em prime labeling} of $G$ is a labeling of its vertices with
distinct integers from $\{1, 2,\ldots ,n\}$ in such a way that the labels
of any two adjacent vertices are relatively prime. We say that $G$ is {\em prime} if it has a prime labeling.
The {\em coprime graph of integers} (see \cite[Section 7.4]{ps}) has the set of all integers as vertex set where two vertices
are adjacent if and only if they are relatively prime.
So, for an $n$-vertex graph, being prime  is equivalent to being a subgraph of the induced subgraph on $\{1, 2,\ldots ,n\}$  by the coprime graph of integers.
Many properties of the coprime graph of integers including investigating its subgraphs were studied by Ahlswede and Khachatrian \cite{1, 2, 3}, Erd\H{o}s \cite{7, 8, 9}, Erd\H{o}s, S\'ark\"ozy, and Szemer\'edi \cite{10, 12}, Szab\'o and T\'oth \cite{30}, Erd\H{o}s and S\'ark\"ozy \cite{11, 13},
 and S\'ark\"ozy \cite{28}. For a survey on the known results on this subject see \cite{ps}.

The notion of  prime labeling originated with Entringer and was introduced in \cite{t}.
Entringer around 1980 conjectured that all trees are prime. 
Haxell, Pikhurko, and Taraz \cite{h} proved that there is an integer $n_0$ such all trees with at least $n_0$ vertices are prime.
Besides that, several classes of graphs have been shown to be prime, see \cite{g} for more details.

The prime labeling of $r$-regular graphs have been studied so far for $r\le3$.
For $r=3$, i.e., for cubic graphs,  Schroeder \cite{sch} confirmed a conjecture of Schluchter and Wilson \cite{schwil} and classified prime cubic graphs: a cubic graph $G$ is prime if and only if $G$ is bipartite and $G\ne K_{3,3}$.
This result, in particular provides an additional proof that
the {\em ladder graph} (the Cartesian product $P_n\times P_2$) is prime for all $n\ge 1$, as conjectured by Varkey (see \cite{g,v}) and
first proved by Ghorbani and Kamali \cite{ghk}. Classification of prime $2$-regular graphs has remaind open  to date. A $2$-regular graph $G$ must be
a disjoint union of cycles: $G =C_{n_1}\cup\cdots\cup C_{n_m}$, where each $n_i$ is at least $3$. The
following conjecture was first given by Tout, Dabboucy, and Howalla:
\begin{conj}[\cite{t}] \label{conj:2-reg} Let $G =C_{n_1}\cup\cdots\cup C_{n_m}$ be a $2$-regular graph. Then $G$ is prime if and only if
at most one $n_i$ is odd.
\end{conj}
Note that if a graph $G$ with $n$ vertices is prime, the vertices with even labels form
an independent set. Thus its independence number satisfies $\al(G)\ge\lfloor n/2\rfloor$.
For cycles, we have $\al(C_\ell) =\lfloor\ell/2\rfloor$. Consequently, if in $G =C_{n_1}\cup\cdots\cup C_{n_m}$ more than one $n_i$ is odd, then $\al(G)<\lfloor n/2\rfloor$.
  Hence, the necessity
of the condition that ``at most one $n_i$ is odd"  is obvious.
Some partial cases of Conjecture~\ref{conj:2-reg} have been settled in the literature which are reported below:
\begin{itemize}
\item $m\le 4$ (\cite{der});
\item (i) $m\le 7$ provided that all $n_1,\ldots,n_m$ are
even;  (ii) for arbitrary $m$ when $n_1 =\cdots = n_m$ is a sufficiently large even integer (\cite{borosh});
\item (i) if each $n_i$ is even; (ii) if $n_m$ is odd and $\gcd(n_m-1, n) = 1$; (iii)
if $n_m$ is odd and $n_m$ can be written as $2^r + p^s$, for some $r\ge1$ and odd prime $p$ which is
relatively prime to $2^r-1$ (\cite{sch}).
\end{itemize}

Let $n_1,\ldots,n_{m-1}$ be even and $n_m$ be odd. By the above result of \cite{sch}, $C_{n_1}\cup\cdots\cup C_{n_{m-1}}$ is prime.
So to prove Conjecture~\ref{conj:2-reg} it suffices to show that $C_{n_m}$ has a prime labeling with labels $k,\ldots,k+n_m-1$ with $k=n_1+\cdots+n_{m-1}+1$. Note that here both $k$ and $n_m$ are odd.
Therefore, as observed in \cite{sch},  Conjecture~\ref{conj:2-reg} follows from the following conjecture:

\begin{conj}[\cite{sch}]\label{conj:n,k} If  $n,k$ are odd integers and $n\ge3$, then $C_n$ has a prime
labeling using the labels  $k, k + 1, \ldots , k + n - 1$.
\end{conj}

Motivated by Conjecture~\ref{conj:n,k}, we define the {\em $k$-coprime graph} of order $n$, denoted $\cpg(k,n)$ as the graph with vertex set  $\{k, k+1, \ldots, k+n-1\}$ in which two vertices $a,b$ are adjacent if and only if they are coprime, i.e. $\gcd(a,b)=1$.
Here $n$ can be any positive integer and $k$ any integer. If $0$ happen to be a vertex of our graph, it has at most two neighbors, namely $-1$ and $1$, because
for every nonzero integer $a$, $\gcd(a,0)=|a|$.

As the main result of this paper, we characterize  Hamiltonian $k$-coprime graphs as follows.
\begin{thm}\label{thm:main} Let $k$ and $n\ge3$ be integers. Then $\cpg(k,n)$ is Hamiltonian if and only if either
\begin{itemize}
  \item[\rm(i)]  both $n$ and $k$ are odd, or
  \item[\rm(ii)] $n$ is even and  each of $k$ and $k+n-1$ is not divisible by some odd prime less than $n$.
\end{itemize}
\end{thm}
Obviously, from Part (i) of Theorem~\ref{thm:main}, Conjecture~\ref{conj:n,k} and consequently Conjecture~\ref{conj:2-reg} follow.

\section{Proofs}

In this section, we present the proof of Theorem~\ref{thm:main} which is organized as follows: in Theorems~\ref{thm:n,k-odd} and \ref{thm:va(n)not|k} we shall prove that if
$n,k$ satisfy the conditions (i) or (ii) of Theorem~\ref{thm:main}, respectively, then  $\cpg(k,n)$ is Hamiltonian; in Theorem~\ref{thm:notHamilton}, we show that if $n,k$ do not satisfy (i) and (ii), then    $\cpg(k,n)$ is not Hamiltonian.

We start with the following useful lemma.

\begin{lem} \label{lem:edges} Let $G$ be a graph with a Hamiltonian path $v_1,v_2,\ldots, v_n$. If there is a sequence of indices $1<i_1<\cdots<i_k<n$ such that
$G$ contains the edges
\begin{equation}\label{eq:edges}
v_1v_{i_1+1},v_{i_1}v_{i_2+1},\ldots,v_{i_{k-1}}v_ {i_k+1},v_{i_k}v_n,\end{equation}
then $G$ is Hamiltonian.
\end{lem}
\begin{proof}{Consider the induced subgraph of $G$ by the edges of the path $v_1,v_2,\ldots, v_n$ together with the edges given in \eqref{eq:edges}.
If we remove the edges $v_{i_1}v_{i_1+1},v_{i_2}v_{i_2+1},\ldots,v_{i_k}v_{i_k+1}$ from this subgraph, what is left is a cycle with $n$ edges, and so we are done.
}\end{proof}
In what follows, we frequently  use the fact that for any distinct nonzero integer $a,b$,
\begin{equation}\label{eq:gcd}
\gcd(a,b)=\gcd(a,b-a).
\end{equation}

\begin{lem}[\cite{sch}]\label{lem:2^r+p^s} Any odd integer greater than $1$ and less than $149$ can be written as $2^r+p^s$ where $r\ge1$, $s\ge0$, and $p$ is an odd prime with $p\nmid2^r-1$.
\end{lem}

In the next theorem, we prove that if $n$ and $k$ satisfy the condition (i)  of Theorem~\ref{thm:main},  then  $\cpg(k,n)$ is Hamiltonian.

\begin{thm}\label{thm:n,k-odd} Let $k$ and $n\ge3$ be odd integers. Then $\cpg(k,n)$ is Hamiltonian.
\end{thm}
\begin{proof}{Let $G=\cpg(k,n)$ and $k':=k+n-1$. The graph $G$ contains the Hamiltonian path  $k,\ldots, k'$.
We define the sequence $a_0,a_1,\ldots$ as follows. We set $a_0=k$. Assume that $a_{i-1}$ is already defined, we choose $a_i$ in such a way that
$$a_{i-1}< a_i\le k'-1,~~\gcd(a_i+1,a_{i-1})=1,~~\hbox{and $a_i$ is odd}.$$

We assume that $m$ is the largest index for which $a_m$ can be defined.
If $\gcd(a_m,k')=1$, then we have the edges
$$\{a_0,a_1+1\},\{a_1,a_2+1\},\ldots,\{a_{m-1},a_m+1\},\{a_m,k'\}$$
in $G$ and thus we are done by Lemma~\ref{lem:edges}.

Hence we assume that $\gcd(a_m,k')>1$. In what follows, for simplicity  we write $a$ for $a_m$.
First suppose that $k'-a\le32$. Since $k'-a+1$ is an odd integer $\ge3$, by Lemma~\ref{lem:2^r+p^s}, $k'-a+1=2^r+p^s$ for some $r\ge1$, $s\ge0$, and an odd prime $p$ with $p\nmid2^r-1$.
If $s=0$, then $k'=a+2^r$, and $a$ being odd implies that $\gcd(a,k')=1$ which is not the case. Hence $s\ge1$.
Note that we have either $\gcd(a,a+p^s)=1$ or $\gcd(k',a+2^r-1)=1$, since otherwise we have $\gcd(a,a+p^s)>1$ and so  $\gcd(a,p^s)>1$  which implies that $p\mid a$. Also we have $1<\gcd(k',a+2^r-1)\mid k'-(a+2^r-1)=p^s$ which in turn implies that $p\mid k'$.
It turns out that $p$ divides $k'-a=2^r+p^s-1$ and so $p\mid 2^r-1$, a contradiction.
Now, if $\gcd(a,a+p^s)=1$, then $a_{m+1}$ can be defined as $a_{m+1}=a+p^s-1$ which is not possible by our choice of $m$. Therefore,
  $\gcd(k',a+2^r-1)=1$. It follows that we have the edges
$$\{k,a_1+1\},\{a_1,a_2+1\},\ldots,\{a_{m-1},a+1\},\{a,a+2^r\},\{a+2^r-1,k'\}$$
in $G$ and again we are done by Lemma~\ref{lem:edges}.

Next,  suppose that $k'-a\ge33$. Note that for every odd prime $p<k'-a$, we have $p\mid a$ (since if there is an odd prime $p<k'-a$ with $p\nmid a$, then $a_{m+1}$ can be defined as $a_{m+1}=a+p-1$, a contradiction).
As $a$ is odd, by \eqref{eq:gcd} it is seen that all the integers
\begin{equation}\label{eq:a+2}
a+2,a+4,a+8,a+16,a+32
\end{equation}
 are coprime to $a$ and so they are coprime to every prime $p<k'-a$. It follows that any integer $b$ of the list \eqref{eq:a+2} is coprime to all positive integer less than $k'-a$.
 Now, if $a+1\le c\le k'$ and $c\ne b$, then $1\le|b-c|<k'-a$. So, in view of \eqref{eq:gcd}, $\gcd(b,c)=\gcd(b,|b-c|)=1$. Therefore,
  $b$ is coprime to all the integers in
 $\{a+1,a+2,\ldots,k'\}\setminus\{b\}$.
 Similarly, using \eqref{eq:gcd}, we see that $a+1$ is coprime to all  odd integers in the range $a,\ldots,k'$.
 It turns out that $G$ contains a path $P$ on the vertices
 $$ a,a+1,k',k'-1,k'-2,\ldots,a+33,a+2,a+3,a+32,a+31,\ldots,a+4.$$
 If $a=k$, the path $P$ together with the edge $\{a,a+4\}$ give rise to a Hamiltonian cycle of $G$.
 Otherwise, since $\gcd(a-1,a)=1$ and $5\mid a$, we have $5\nmid a-1$ and thus $\gcd(a-1,a+4)=1$. This shows that
 $P$ together with the path $a+4,a-1,a-2,\ldots,k+1,k$ give a Hamiltonian path of $G$. Since we have the edges
 $\{k,a_1+1\},\{a_1,a_2+1\},\ldots,\{a_{m-1},a_m\}$ in $G$, in view of Lemma~\ref{lem:edges}, it follows that $G$ is Hamiltonian.
}\end{proof}

By $\va(n)$, we denote the product of all odd primes less than $n$. This function has a role  in Hamiltonicity of
$\cpg(k,n)$.

\begin{thm}\label{thm:notHamilton} In the following cases, $\cpg(k,n)$ is not Hamiltonian:
\begin{itemize}
  \item[\rm(i)] $n$ odd and $k$ even;
  \item[\rm(ii)] $n$ even and either  $\va(n)\mid k$ or $\va(n)\mid k+n-1$.
\end{itemize}
\end{thm}
\begin{proof}
 (i) If $n$ is odd and $k$ even, then $\cpg(k,n)$ has an independent set of size $(n+1)/2$ consisting of even vertices. Note if a graph with $n$ vertices has an independent set of size larger than $n/2$, then it cannot be Hamiltonian. So $\cpg(k,n)$ is not Hamiltonian in this case.

 (ii) Let $n$ be even. Then $G=\cpg(k,n)$  has an independent set of size $n/2$ consisting of even vertices.
 It follows that if $G$ is Hamiltonian, then in any Hamiltonian cycle of $G$, every edge should join two vertices with opposite parities. In particular,  any vertex of $G$ should have at least two neighbors with opposite parity.

 First assume that  $\va(n)\mid k$. We show that in this case, the vertex $k$ has only one neighbor with opposite parity, namely $k+1$, which in turn implies that $G$ is not Hamiltonian.
 To see this, let $\ell$ be a neighbor of $k$ with opposite parity and $k+1<\ell\le n+k-1$. So
$\ell-k$ is an odd integer with $3\le\ell-k\le n-1$. Hence, there is some odd prime $p$ such that $p\mid\ell-k$. As $p<n$, we have $p\mid\va(n)$ and thus $p\mid k$. It follows that $p\mid\ell$, too. So $k$ and $\ell$ cannot be adjacent, a contradiction.

In the case $\va(n)\mid k+n-1$,  with a similar proof as given above, we see that the vertex $k+n-1$ has only one neighbor with opposite parity. Thus, $G$ cannot be Hamiltonian.
\end{proof}

Here, we give another property of $\va(n)$  in connection with Hamiltonicity of
$\cpg(k,n)$.

\begin{lem}\label{lem:k+va(n)} Let $n\ge4$ be even. Then,
$v_1,\ldots,v_n,v_1$ is a Hamiltonian cycle of $\cpg(k,n)$ if and only if $v_1+\va(n),\ldots,v_n+\va(n),v_1+\va(n)$ is a Hamiltonian cycle of $\cpg(k+\va(n),n)$.
\end{lem}
\begin{proof} Since $n$ is even, as it is observed in the proof Theorem~\ref{thm:notHamilton}, in any Hamiltonian cycle of a coprime graph of order $n$, the ends of every edge have opposite parities.
 Let $a,b$ be two integers with opposite parities and $k\le a<b\le k+n-1$. We claim that $\gcd(a,b)=1$ if and only if $\gcd(a+\va(n),b+\va(n))=1$, from which the result follows.
 To see this, suppose that
$\gcd(a,b)=1$ and an odd prime $p$ divides $\gcd(a+\va(n),b+\va(n))$. Then $p\mid(b+\va(n)-a-\va(n))=b-a$. Since $b-a<n$, $p\mid\va(n)$ and thus $p$ should divide both $a,b$ and so $p=1$, a contradiction. Thus $\gcd(a+\va(n),b+\va(n))=1$. The other direction is similar.
\end{proof}

We need further properties of $\va(n)$.

\begin{lem}\label{lem:var}
 For every integer $n\ge6$,  $\va(n)\ge2n+1$.
\end{lem}

\begin{proof}{For $n=6,7$, the inequality holds: $\va(6)=\va(7)=15$.
 We first verify by induction that $\va(2^i)>
 2^{i+2}$ for $i\ge3$. For $i=3$, we have $\va(8)=3\cdot5\cdot7>32$. By the Bertrand's postulate,
  there is a prime $p$ with $2^i<p<2^{i+1}$. Therefore,  by induction we have
$$\va(2^{i+1})\ge p\cdot\va(2^i)>p\cdot2^{i+2}>2^{i+3}, \quad\hbox{for $i\ge3$.}$$
Now, for any integer $n\ge8$, choose $i\ge3$ such that $2^i\le n<2^{i+1}$. Then
$$\va(n)\ge\va(2^i)>2^{i+2}>2n+1.$$
}\end{proof}

We use the standard notation $\pi(n)$ to denote the number of primes less than $n$. 
\begin{lem}\label{lem:pi>t+3} Let $n\ge12$ be even, and $t$ be the number of prime factors of $n-1$. Then $\pi(n)\ge t+4$.
\end{lem}
\begin{proof} We proceed by induction on $t$. If $t=1$, we are done as
$\pi(n)\ge\pi(12)=5.$
If $t=2$, we have $n\ge16$ and thus $\pi(n)\ge\pi(16)=6$. If $t=3$, then  $\pi(n)\ge\pi(n-1)\ge\pi(2\cdot3\cdot5)>7.$
Let $p_1,p_2,\ldots$ be the sequence of primes. Then $\pi(p_1\cdots p_t)\ge\pi(p_1\cdots p_{t-1})+1$ because there is a prime  between $p_1\cdots p_{t-1}$ and $p_1\cdots p_t$ (as a consequence of the Bertrand's postulate).
Now, for $t\ge4$, if $n-1$ has $t$ prime factors, then  by the induction hypothesis,
 $$\pi(n)\ge\pi(n-1)\ge\pi(p_1\cdots p_t)\ge\pi(p_1\cdots p_{t-1})+1\ge t+4.$$
\end{proof}

\begin{lem}\label{lem:opp-parity} Let $k$ and $n\ge5$ be odd integers. Then in the graph $\cpg(k,n)$, at least one of $k$ or $k+n-1$ has at least two neighbors with opposite parity.
\end{lem}
\begin{proof} The vertices $k$ and $k':=k+n-1$ have the neighbors $k+1$ and $k'-1$, respectively.
So it is enough to show that either of $k$ or $k'$ have some other neighbor with opposite parity.
 Since the first and the last vertex of $G:=\cpg(k,n)$ are odd, $G$ has $t$ even and $t+1$ odd vertices, for some $t\ge2$.
By Theorem~\ref{thm:n,k-odd}, $G$ has a Hamiltonian cycle $C$. Since even vertices form an independent set of $G$, $2t$ edges of $C$ are between even vertices and odd vertices. Hence $C$ has a unique edge $e$ whose ends are both odd. If either of $k$ or $k'$ is not on $e$, we are done.
Therefore, suppose that $e=\{k,k'\}$. So, by \eqref{eq:gcd}, $\gcd(k,n-1)=1$.
If $n-1$ has an odd factor $p$, then $k+p$ is adjacent with $k$ and we are done. Otherwise, $n-1$ is a power of $2$.
Therefore, we have either $3\nmid k$ or $3\nmid k'$  (since otherwise $3\mid(k'-k)=n-1$, a contradiction).
Hence, either $k$ is adjacent with $k+3$ or $k'$ is adjacent with $k'-3$, as desired.
\end{proof}

Finally,  we prove that if $n,k$ satisfy the condition (ii)  of Theorem~\ref{thm:main},  then  $\cpg(k,n)$ is Hamiltonian.
This together with Theorems~\ref{thm:n,k-odd} and \ref{thm:notHamilton} complete the proof of Theorem~\ref{thm:main}.

\begin{thm}\label{thm:va(n)not|k} Let $n\ge4$ be even. If $\va(n)\nmid k $ and $\va(n)\nmid k+n-1$, then  $\cpg(k,n)$ is Hamiltonian.
\end{thm}
\begin{proof} Let $k':=k+n-1$ and $G:=\cpg(k,n)$.
Since $\va(n)\nmid k$, there is an odd prime $p<n$ such that $p\nmid k$. It follows that $k+p$ is adjacent to $k$ in $G$.
Let $\bar p$ be the smallest such prime and $\ell:=k+\bar p$.
Indeed, $\ell$ is the smallest neighbor of $k$ with opposite parity other than $k+1$.
Similarly, we can define $\bar p'$ and $\ell':=k'-\bar p '$ as the largest neighbor of $k'$ with opposite parity other than $k'-1$.
 We may assume that
\begin{equation}\label{eq:l-k>l'-k'}
\bar p'\le\bar p\quad\hbox{or equivalently}\quad k'-\ell'\le\ell-k,
\end{equation}
since otherwise we can consider $\cpg(-k',n)$ instead which is isomorphic to $G$.
Furthermore, we may assume that $k'$ is odd, otherwise, by Lemma~\ref{lem:k+va(n)}, we can consider the graph $\cpg(k+\va(n),n)$ in which the last vertex, i.e. $k'+\va(n)$ is odd.

 By induction, we prove the stronger statement that $G=\cpg(k,n)$ contains a Hamiltonian cycle including the edges $\{k,k+1\}$ and $\{k'-1,k'\}$.
 We call such a Hamiltonian cycle {\em special}.
At first, we need to prove the statement for $n\le10$.
\begin{itemize}
\item[$n=4$:]  Since $3=\va(4)\nmid k$, we have the special Hamiltonian cycle $k,k+1,k+2,k+3,k$ in $G$.
\item[$n=6$:] We have $3\cdot5=\va(6)\nmid k$ and $\va(6)\nmid k'= k+5$ . If $5\nmid k$, then we have the special Hamiltonian cycle $k,\ldots,k+5,k$ in $G$.
If $5\mid k$, then necessarily $3\nmid k$, $3\nmid k'$ and thus $k,k+1,k+2,k+5,k+4,k+3,k$ is a special Hamiltonian cycle.

\item[$n=8$:] We have $3\cdot5\cdot7=\va(8)\nmid k$ and $\va(8)\nmid k'=k+7$. If $7\nmid k$, then $k,\ldots,k+7,k$ is a special Hamiltonian cycle of $G$. So let $7\mid k$. Note that $3$ cannot divide both $k,k'$ as $k'-k=7$. By \eqref{eq:l-k>l'-k'}, we may assume that $3\nmid k'$. If further
$5\nmid k$, then we have the special Hamiltonian cycle
$k,k+1,k+2,k+3,k+4,k+7,k+6,k+5,k.$  If $5\mid k$, then we have have necessarily $3\nmid  k$, $5\nmid k'$ and thus $k,k+3,k+4,k+5,k+6,k+7,k+2,k+1,k$ is a special Hamiltonian cycle .

\item[$n=10$:]  We have $3\cdot5\cdot7=\va(10)\nmid k$ and $\va(10)\nmid k'=k+9$. If $3\nmid k$, then we have the special Hamiltonian cycle $k,\ldots,k+9,k$ in $G$. So we assume that $3\mid k$ (and so $3\mid k'$). Also, $5$ cannot divide both $k,k'$ and again by \eqref{eq:l-k>l'-k'},
we can assume that $5\nmid k'$. If further $5\nmid k$, then we have the special Hamiltonian cycle $k,k+1,k+2,k+3,k+4,k+9,k+8,k+7,k+6,k+5,k$. If $5\mid k$, since $3\mid k$, too, then necessarily $7\nmid k$.
If further $7\nmid k+2$, then we have the  special Hamiltonian cycle
$k,k+1,k+2,k+9,k+8,k+3,k+4,k+5,k+6,k+7,k$.
If $7\mid k+2$, since we already have $15\mid k$, then $k\equiv75\pmod{105}$.
By Lemma~\ref{lem:k+va(n)}, we only need to consider $1\le k\le\va(10)=105$,
  so we may assume that $k=75$ in which case
 $75,76,81,80,77,78,83,84,79,82,75$ is a special Hamiltonian cycle.
\end{itemize}

In what follows, we assume that $n\ge12$. We consider two cases.

\noindent{\bf Case 1.} $\ell'>\ell$.

Let $n'=\ell'-k+2$ which is even as $k\equiv\ell'\pmod2$.

First, assume that $\va(n')\nmid\ell'+1$. Since $k$ has a neighbor with opposite parity between $k+1$ and $\ell'+1$, we have
$\va(n')\nmid k$ (if $\va(n')\mid k$, then $\gcd(k,k+r)>1$ for every odd $r$, $1<r<n'$).
 Hence $\cpg(k,n')$ satisfies the induction hypothesis and so it has a special Hamiltonian cycle.
 In particular, $\cpg(k,n')$ has a Hamiltonian path $P$ between $\ell'$ and $\ell'+1$, where $P$ includes the edge $\{k,k+1\}$.
Now, if we let $P'$ be the path $\ell'+1,\ell'+2,\ldots,k'-1,k',\ell'$, then $P\cup P'$ gives rise to a Hamiltonian cycle of $G$  including both the edge $\{k,k+1\}$ and  $\{k'-1,k'\}$.

Next, assume that $\va(n')\mid\ell'+1$. Consider $\bar p'=k'-\ell'$ which is an odd prime.
We claim that $\bar p'=3$. For a contradiction, assume that $\bar p'\ge5$. Consider the graph $G'':=\cpg(\ell'+1,\bar p')$.
In $G''$ the smallest vertex $\ell'+1$ is odd and the largest vertex $k'$ has only one neighbor with opposite parity, namely $k'-1$.
 Thus, by Lemma~\ref{lem:opp-parity}, $\ell'+1$ has at least one neighbor with opposite parity other than $\ell'+2$. Suppose $\ell'+1+r$ is this neighbor, i.e.
 $\gcd(\ell'+1,\ell'+1+r)=1$.  Here $r$ should be an odd integer $\ge3$ and further  by \eqref{eq:l-k>l'-k'},
  $$r<k'-\ell'\le\ell-k<\ell'-k<n'.$$
  We may assume that $r$ is a prime since otherwise $r$ can be replaced by any of its prime factors.
  Thus $r\mid\va(n')$ and so $r\mid\ell'+1$. This yields  $r\mid\ell'+1+r$. By this contradiction, the claim follows, that is $\bar p'=3$. Therefore, $\ell'+1=k'-2$. So $\va(n')\mid k'-2$.
It turns out that for any prime $p<n'$, we have $p\nmid k'$.
Given that $$k'-(\ell-1)=\ell'-\ell+4\le\ell'-(k+3)+4<n',$$ it follows that $\gcd(\ell-1,k')=1$, and thus we have the following special Hamiltonian cycle in $G$:
$$k,k+1,\ldots,\ell-1,k',k'-1,\ldots,\ell,k.$$

\noindent{\bf Case 2.} $\ell'<\ell$ (having opposite parities, $\ell=\ell'$ is not possible).

We claim that if $k,k'$ have neighbors $a,a'$, respectively, with $a\not\equiv a'\pmod 2$ and $a-a'\ge5$, then $G$ has a special Hamiltonian cycle.
To see this, let $m:=a-a'+1$ which is an even integer $\ge6$.  We observe that $\va(m)$ does not divide either both of $a+1,a'+1$, or both of
$a-1,  a'-1$ (if this does not hold,  $\va(m)$ should divide at least one of the integers $2, m-3, m-1, m+1$ which is a contradiction in view of Lemma~\ref{lem:var}).
If $\va(m)\nmid a+1,a'+1$, then from the induction hypothesis, it follows that the graph $\cpg(a'+1,m)$ contains a Hamiltonian path $P$ between $a$ and $a+1$.
This together with the path
$$a+1,a+2,\ldots,k'-1,k',a',a'-1,\ldots,k+1,k,a,$$
 give rise to a special Hamiltonian cycle in $G$.
 If $\va(m)\nmid a-1,a'-1$, then from the induction hypothesis, it follows that the graph $\cpg(a'-1,m)$ contains a Hamiltonian path $P$ between $a'$ and $a'-1$.
This together with the path
$$a'-1,a'-2,\ldots,k,a,a+1,\ldots,k'-1,k',a'$$
 give rise to a special Hamiltonian cycle in $G$
  and thus the claim follows.

Now, if $\ell'=\ell-1$, then $G$ contains
the following special Hamiltonian cycle:
$$k,\ell,\ell+1,\ldots,k'-1,k',\ell',\ell'-1,\ldots,k+1,k.$$
 So we assume that $\ell-\ell'\ge3$.  Let $t$ be the number of prime factors of $n-1$. Since $n\ge12$,
by Lemma~\ref{lem:pi>t+3}, $\pi(n)\ge t+4$. Hence the is a prime $q<n$ such that $q\nmid n-1$ and $q\not\in\{2,\bar p,\bar p'\}$.
As $k'-k=n-1$, it follows that either $q\nmid k$ or $q\nmid k'$. If $q\nmid k$, then $a=k+q$ is a neighbor of $k$ with opposite parity other than $\ell=k+\bar p$. By the definition of $\ell$, we must have
$a\ge\ell+2$. Therefore, $a-\ell'\ge5$ and we are done in view of the above claim.  If $q\nmid k'$, we are done in a similar manner.
\end{proof}

\section*{Acknowledgements}
The second author carried this work during a Humboldt Research Fellowship at the University of Hamburg. He thanks the Alexander von Humboldt-Stiftung for financial support.

{}

\end{document}